\documentclass{amsart}

\usepackage[skip=5pt plus1pt, indent=20pt]{parskip}

\usepackage{graphicx} % Required for inserting images

\usepackage[utf8]{inputenc}
\usepackage{amssymb} 
\usepackage{amsmath} 
\usepackage{amscd}
\usepackage{amsbsy}
\usepackage{comment}
\usepackage[matrix,arrow]{xy}
\usepackage{hyperref}
\usepackage{float}
\usepackage[utf8]{inputenc}
\usepackage{xcolor}
\usepackage{enumerate}
\usepackage{booktabs}
\usepackage{makecell}
\usepackage{tabularx}
\usepackage{breqn}
\usepackage{multirow}

\newcommand{\Z}{\mathbb{Z}}
\newcommand{\Q}{\mathbb{Q}}

\begin{document}

\newtheorem{theorem}{Theorem}
\newtheorem{lemma}{Lemma}[section]
\newtheorem{proposition}[lemma]{Proposition}
\newtheorem{algorithm}[lemma]{Algorithm}
\newtheorem{corollary}[lemma]{Corollary}
\newtheorem{conjecture}{Conjecture}

\theoremstyle{definition}
\newtheorem{definition}[theorem]{Definition}

\theoremstyle{remark}
\newtheorem{remark}[theorem]{Remark}

\newtheorem{acknowledgment}{Acknowledgement}

\title[]{Perfect codes over non-prime power alphabets: an approach based on Diophantine equations}

\author{Pedro-Jos\'{e} Cazorla Garc\'{i}a}

\address{Department of Mathematics, University of Manchester, Manchester, United Kingdom, M13 9PL}
\email{pedro-jose.cazorlagarcia@manchester.ac.uk}

\date{\today}

\begin{abstract}
Perfect error correcting codes allow for an optimal transmission of information while guaranteeing error correction. For this reason, proving their existence has been a classical problem in both pure mathematics and information theory. Indeed, the classification of the parameters of $e-$error correcting perfect codes over $q-$ary alphabets was a very active topic of research in the late 20th century. Consequently, all parameters of perfect $e-$error correcting codes were found if $e \ge 3$, and it was conjectured that no perfect $2-$error correcting codes exist over any $q-$ary alphabet, where $q > 3$. In the 1970s, this was proved for $q$ a prime power, for $q = 2^r3^s$ and for only $7$ other values of $q$. Almost $50$ years later, it is surprising to note that there have been no new results in this regard and the classification of $2-$error correcting codes over non-prime power alphabets remains an open problem. In this paper, we use techniques from the resolution of generalised Ramanujan--Nagell equation and from modern computational number theory to show that perfect $2-$error correcting codes do not exist for $172$ new values of $q$ which are not prime powers, substantially increasing the values of $q$ which are now classified. In addition, we prove that, for any fixed value of $q$, there can be at most finitely many perfect $2-$error correcting codes over an alphabet of size $q$.
\end{abstract}


\maketitle

\section{Introduction}

\subsection{Background}
Error correcting codes have been widely studied since the 1940s. Apart from their intrinsic mathematical interest, this is due to the fact that error correcting codes have many useful applications for information transmission and engineering. For example, in \cite{IBM} the authors discuss possible applications of error-correcting codes to semiconductor memory, while the classical article by Shannon \cite{Shannon} introduced error-correcting codes as a way to transmit information over noisy channels.

%, both due to their mathematical interest and for their applicability in the transmission of information over noisy channels. 

While these applications are one of the main reasons why error correcting codes are studied, the developments in this paper will be theoretical. Throughout the article, we shall consider the following standard definition of an error correcting code.

\begin{definition}    
    Let $q \ge 2$, $e \ge 1$ be positive integers, and let $Z_q$ be a set with $q$ elements, which we shall call the \emph{alphabet}. An \emph{$e-$error correcting code} $C$ with parameters $(n, M)$ is a subset of $Z_q^n$ of size $M$ and such that any two elements $\overline{w}_1, \overline{w}_2 \in C$ differ in at least $2e+1$ positions.
\end{definition}

It is a standard fact in coding theory (for example, see \cite[Theorem 2.16]{Hill}) that any $e-$error correcting code with parameters $(n, M)$ over an alphabet of size $q$ satisfies the \emph{sphere packing bound} or \emph{Hamming bound}:
\begin{equation}
    \label{eqn:inequality}
M\left(\sum_{j = 0}^e {\binom{n}{j}}(q-1)^j\right) \le q^n.
\end{equation}
Let us briefly explain why this inequality holds. For this, let us define the set% from, we define the set
\begin{equation}
    \label{eqn:eq1}
B(\overline{w}, d) = \{\overline{w}_2 \in C \mid \overline{w} \text{ and } \overline{w}_2 \text{ differ in at most } d \text{ positions.}\},
\end{equation}
where $\overline{w} \in C$ is fixed and $d \ge 0$ is an integer. It is an easy combinatorial exercise to see that 
\begin{equation}
    \label{eqn:eq2}
|B(\overline{w}, d)| = \sum_{j = 0}^d {\binom{n}{j}}(q-1)^j. 
\end{equation}

In order for a code to be an $e-$error correcting code, it follows that the sets $B(\overline{w}, e)$ and $B(\overline{w}', e)$ are disjoint for $\overline{w} \neq \overline{w}'$, and therefore, we have that
\begin{equation}
    \label{eqn:eq3}
\left|\bigcup_{\overline{w} \in C} B(\overline{w}, e)\right| = M\left(\sum_{j = 0}^d {\binom{n}{j}}(q-1)^j\right). 
\end{equation}
Since there are precisely $q^n$ elements in $Z_q^n$, inequality \eqref{eqn:inequality} follows. In the case where \eqref{eqn:inequality} is an equality, we say that the code is \textbf{perfect}. In other words, $C$ is perfect if and only if $n$ and $M$ satisfy the following Diophantine equation:
\begin{equation}
    \label{eqn:perfectcondition}
    M = {q^n}\left/\sum_{j = 0}^e {\binom{n}{j}}(q-1)^j\right..
\end{equation}
Perfect codes are optimal for information transmission, in the sense that they allow for the maximum possible number of words while being able to correct $e$ errors. In other words, perfect codes are the most expressive $e-$error correcting codes that are theoretically possible, and it is desirable to know whether they exist for a fixed value of $q$ and $e$.

For these reasons, the classification of perfect codes was a very active topic of research in the late 20th century, and this led to the production of a vast amount of literature on the topic. The most relevant works are summarised in Table \ref{tab:comparison} and we refer the reader to \cite{survey} for a detailed survey on the history of perfect codes.

\begin{table}[!ht]
    \centering 
    \begin{tabular}{cccc}
    \toprule 
    Reference & $q$ & $e$ & Main results \\ \midrule 
    \makecell[c]{Tietäväinen \cite{primepowers}, and \\Leontiev and Zinoniev \cite{primepowers2}  \\(independently)} & \makecell{Prime \\ power} & $\ge 1$& \makecell{All perfect codes over prime power \\ alphabets are either Hamming \\ codes or have parameters $(n, M, $ \\ $e, q) = (11, 3^6, 2, 3), (23, 2^{12}, 3, 2)$.} \\
    \midrule 
    Reuvers \cite{Reuvers} & \multirow{5}{*}{\makecell{Non-\\prime \\ power}} & 3, 4, 5 & \multirow{5}{*}{\makecell{There are no perfect $e-$error \\correcting codes over a $q$-ary \\ alphabet.}} \\
    & & & \\
    Best \cite{Best} & & \makecell{7, $\ge 9$} & \\
    & & & \\
    Hong \cite{Hong} & & \makecell{$6$, $8$} \\
    \midrule 
    Reuvers \cite{Reuvers} & \makecell{6, 15, 21, 22, \\ 26, 30, 35} & \multirow{5}{*}{2} & \multirow{5}{*}{\makecell{There are no perfect $e-$error \\correcting codes over a $q$-ary \\ alphabet.}} \\
    & & & \\
    van Lint \cite{vanLint} & 10 & & \\
    & & & \\
    Bassalygo et al. \cite{Bassalygo} & $2^r\cdot 3^s$ & & \\
    \midrule
    \end{tabular}
    \caption{Summary of existing results on perfect $e-$error correcting codes over $q-$ary alphabets}
    \label{tab:comparison}
\end{table}
From Table \ref{tab:comparison}, we immediately notice that perfect codes do exist if $q$ is a prime power, and are either Hamming codes with parameters 
\begin{equation}
    \label{eqn:eq4}
(n, M, e, q) = \left(\frac{q^r-1}{q-1}, \frac{q^r-1}{q-1}-r-1, 1, q\right),
\end{equation}
for any $r \ge 2$, or are \emph{Golay codes} with parameters
\begin{equation}
    \label{eqn:eq5}
(n, M, e, q) = (11, 3^6, 2, 3), (23, 2^{12}, 3, 2).
\end{equation}
This was proved by Tietäväinen \cite{primepowers}, and independently by Leontiev and Zinoniev \cite{primepowers2}. Therefore, if $q$ is a prime power, perfect codes do exist and are completely classified for any value of $e \ge 1$.

However, as it is apparent from Table \ref{tab:comparison}, the situation when $q$ is not a prime power is quite different. In this case, all the existing results are negative for any values of $e$ and $q$. This has led many authors to conjecture (see for example \cite[Section 5]{survey}) the following:
\begin{conjecture}
    \label{conj:nomore}
    Let $q \ge 6$ be a non-prime power. Then, there are no perfect error correcting codes over $Z_q$.
\end{conjecture}

In Table \ref{tab:comparison}, we see that the combination of the works of Reuvers \cite{Reuvers}, Best \cite{Best} and Hong \cite{Best} successfully proves Conjecture \ref{conj:nomore} for $e-$error correcting codes, where $e \ge 3$. For $e \le 2$, the number of results is much more scarce, and almost non-existent for $e=1$, as remarked by Heden in \cite[Section 2]{survey}. 

For $e=2$, as shown in the last rows of Table \ref{tab:comparison}, the existing results are due to Reuvers \cite{Reuvers}, van Lint \cite{vanLint} and Bassalygo et al. \cite{Bassalygo}. Together, they show that perfect $2-$error correcting codes do not exist for
\begin{equation}
    \label{eqn:eq6}
q = 6, 10, 15, 21, 22, 26, 30, 35 \quad \text{and} \quad q = 2^s3^r,
\end{equation}
where $r, s > 0$ are positive integers. We remark that the methods used in all these papers are specific to the cases at hand and cannot be generalised to other values of $q$. This is probably the reason why there have been no new results in the area in more than $50$ years. 

Indeed, we believe that the great difficulty of dealing with the case where $q$ is not a prime power, as well as the lack of a general methodology for these cases justifies the scarcity of literature. 

To illustrate this, we note that there is a significant amount of modern articles which aim to generalise the results in Table \ref{tab:comparison} to extended codes or to different metrics, but they always treat the case when $q$ is a prime power. For example, under this assumption Li and Xing \cite{quantum} classify perfect quantum codes and Gubitosi, Portela and Qureshi \cite{NRT} show analogous results for the Niederreiter-Rosenbloom-Tsfasman metric. 

In conclusion, we identify a clear gap in the literature in the case $e=2$, since, apart from the family studied in \cite{Bassalygo}, there are only $7$ values of $q$ which are not prime powers and for which non-existence has been proved, and there have been no new results in half a century. In addition, the techniques used in the proofs cannot be generalised, which makes it difficult to extend the results to different codes or metrics. 

In this paper, we address this gap in the literature by proving non-existence results for the case $e=2$ and 172 new values of $q$, therefore substantially increasing the evidence towards Conjecture \ref{conj:nomore}. In addition, our methods are theoretically valid for all values of $q$ and the only obstructions that we encounter are of a computational nature. Consequently, it seems feasible to generalise them to other problems, and we shall comment more on this in Section \ref{Sec:futurelines}.

\subsection{The main results} 

For the rest of the paper, we let $q \ge 6$ be a positive integer which is not a prime power. In this paper, we will prove new results related to the non-existence of perfect $2-$error correcting codes over alphabets of size $q$.

% methodology does not require $q$ to be a prime power, so we shall state the theorems in full generality. Of course, if $q$ is a prime power, our results are consistent with those of Tietäväinen \cite{primepowers} and Leontiev and Zinoniev \cite{primepowers2}.

Since our methodology applies for $q= 6, 15, 21, 22, 26, 30, 35$, for $q=10$, and for some cases where $q=2^r\cdot 3^s$, we shall recover the aforementioned results by Reuvers \cite{Reuvers}, van Lint \cite{vanLint} and Bassalygo et al \cite{Bassalygo}. Therefore, our findings are consistent with the previously existing literature.% on non-existence of $2-$error correcting perfect codes over $q-$ary alphabets.

Our main result is the following, which covers all but two values of $q \le 200$, as well as some values of $q$ in the interval $q \in [201,600]$ subject to certain additional restrictions.

%{\color{red} Computations are all over the place. As of right now, there are five relevant files: \textbf{q200.out} (for all $q \le 200$ with $q \not \equiv 9, 13 \pmod{16}$, it has finished but cannot cover $q = 94, 166$), \textbf{missingq200.out} (for all $q \le 200$ with $q \equiv 9, 13 \pmod{13}$, finished with no issues), \textbf{smallPrimes.out} (for $q \le 2000$ supported only on $2, 3, 5, 7, 11$ with $q \not \equiv 9, 13 \pmod{16}$ - this is taking way too long and I am thinking of cutting it at $q \le 500$ - it runs fine until $q \le 600$ so maybe just do whatever, but we need to do $q \equiv 9,13 \pmod{13}$), \textbf{missingq500.out} (for $q \le 500$ with $q \equiv 9,13 \pmod{16}$, already finished) and \textbf{q1000.out} (for $q \le 1000$ with two prime factors - at the moment this is at $q = 422$ but it is failing a lot - probably it does not make sense to include it).}

\begin{theorem}
    \label{thm:main}
    %{\color{red} This should be revised once computations are finished.}
    Let $q \ge 6$ be a positive integer which is not a prime power, satisfying any of the following conditions. %Suppose that either of the following  conditions are satisfied:

    \begin{enumerate}
        \item $q \le 200$ and $q \neq 94, 166$.
        %\item $q \le ?$ and $q$ has precisely two prime divisors.
        \item $q \le 600$ and all prime divisors of $q$ are contained in the set $\{2, 3, 5, 7, 11\}$.
    \end{enumerate}
    Then, there are no perfect $2-$error correcting codes over alphabets of size $q$.
\end{theorem}

The values of $q$ in the theorem are optimal given the current existing computational techniques and the methods used in this paper. We will comment more on this in Section \ref{Sec:limitations}.

The second result that we present applies to all values of $q$, but, instead of showing non-existence of perfect codes, it shows that there can only be finitely many perfect $2-$error correcting perfect codes over a $q-$ary alphabet. This is consistent with Conjecture \ref{conj:nomore}, albeit much weaker.
%We note that this is consistent with Conjecture \ref{conj:nomore}, since we would expect no perfect codes if $q \ge 4$.

\begin{theorem}
    \label{thm:finitelymany}
    Let $q \ge 6$ be a positive integer which is not a prime power. Then, there are at most finitely many perfect $2-$error correcting codes over an alphabet of size $q$.
\end{theorem}

%The second result shows that there are no perfect $2-$error correcting codes for two infinite families of $q$.
%\begin{theorem}
 %   Let $q \ge 6$ be a positive integer. Suppose furthermore that either $q \equiv 9 \pmod{16}$ or $q \equiv 13 \pmod{16}$. Then, there are no perfect $2-$error correcting codes over alphabets of size $q$.
%\end{theorem}

In order to prove Theorems \ref{thm:main} and \ref{thm:finitelymany}, we will study the perfect code condition \eqref{eqn:perfectcondition} with $e=2$. We shall see that, for a fixed value of $q$, it can be reduced to a Ramanujan--Nagell type equation, which we can solve effectively with the use of several techniques from computational number theory. 

The history of Ramanujan--Nagell type equations is rich and has motivated a vast amount of research in the number theory community ever since Ramanujan conjectured that the Diophantine equation
\begin{equation}
    \label{eqn:RNintro}
x^2 + 7 = 2^n,
\end{equation}
only had solutions for $n = 3, 4, 5, 7$ and $15$. This was proved by Mordell in 1948 \cite{Mordell} and led to many researchers considering generalisations of \eqref{eqn:RNintro}. We refer the reader to \cite{RNsurvey} for a detailed survey of generalisations of the Ramanujan--Nagell equation.

In order to prove Theorems \ref{thm:main} and \ref{thm:finitelymany}, we shall solve a Ramanujan--Nagell type equation. If, after resolving the equation, there are any solutions $(n, M)$ to \eqref{eqn:perfectcondition}, we can use Lloyd's theorem to prove that they cannot constitute a perfect code.

The organisation of this paper is as follows. In Section \ref{Sec:RN}, we will reduce the Hamming bound to a Ramanujan--Nagell equation and prove Theorem \ref{thm:finitelymany}. In Section \ref{Sec:resolve}, we will build upon the work of von Känel and Matschke \cite{equations} in order to obtain all solutions to the previously obtained Ramanujan--Nagell equations. In Section \ref{Sec:Lloyd}, we will show that the outstanding solutions cannot form perfect codes by using Lloyd's theorem, and therefore prove Theorem \ref{thm:main}. In Section \ref{Sec:limitationswork}, we shall discuss complications arising from our methodology, as well as potential future lines of work. Finally, in Section \ref{Sec:conclusions}, we will briefly summarise our findings and explain how they overcome the shortcomings in the literature.

All computations in this paper have been carried out with the use of \texttt{Magma} \cite{Magma} code. The code is available for the reader in the author's GitHub repository (\href{https://github.com/PJCazorla/perfect-q-ary-codes}{https://github.com/PJ\-Cazorla/perfect-q-ary-codes}) and can be used to check the correctness of the computations.

\section{The Hamming bound and Ramanujan--Nagell type equations}
\label{Sec:RN}
We suppose that there exists a perfect 2 error-correcting $q-$ary code with %words of length $n >1$ and $M$ total words
parameters $(n, M)$. In order to avoid trivial codes, we shall assume that $n \ge 5$ and $M \ge 2$. % where, in order to avoid trivial codes, we assume that $n > 1$. 
Then, it follows from the perfect code condition \eqref{eqn:perfectcondition} with $e=2$ that
\begin{equation}
    \label{eqn:Hamming1}
\left(1 + n(q-1)+\frac{n(n-1)}{2}(q-1)^2\right)M = q^n.
\end{equation}
Let $q_1, \dots, q_k$ denote the prime divisors of $q$. By \eqref{eqn:Hamming1}, $M \mid q^n$ and, consequently, there is a solution $(n, n_1, \dots, n_k) \in \Z^{k+1}$ to the Diophantine equation
\begin{equation}
    \label{eqn:Hamming2}
1 + n(q-1)+\frac{n(n-1)}{2}(q-1)^2 = q_1^{n_1}\dots q_k^{n_k},
\end{equation}
with $n \ge 5$ and where $q \ge 6$ is fixed. Now, let us introduce some notation. Let $x$ and $D_q$ be given by the following expressions:
\begin{equation}
    \label{eqn:definitions}
x = 2n(q-1) + 3 - q \quad \text{and} \quad D_q = 8-(q-3)^2.
\end{equation}
Multiplying \eqref{eqn:Hamming2} by 8 and completing the square, we see that the existence of a perfect 2-error correcting code over a $q-$ary alphabet implies that there exists a solution $(x, n_1, \dots, n_k) \in \Z^{k+1}$ to the Diophantine equation
\begin{equation}
    \label{eqn:RN}
x^2 + D_q = 8q_1^{n_1}\dots q_k^{n_k}.
\end{equation}
Let $A$ denote the set of positive integers supported only on $\{q_1, \dots, q_k\}$, so that 
\begin{equation}
    \label{eqn:eq7}
A = \{q_1^{\alpha_1}\dots q_k^{\alpha_k} \mid \alpha_1, \dots, \alpha_k \ge 0\}.
\end{equation}
Clearly, the Diophantine equation \eqref{eqn:RN} can be rewritten as 
\begin{equation}
    \label{eqn:generalisedRN}
x^2 + b = cy, 
\end{equation}
where $x \in \Z$, $y \in A$, $b = D_q$ and $c= 8$. This is an example of a \textbf{generalised Ramanujan--Nagell equation}, as considered by von Känel and Matschke in \cite{equations}. We can immediately use their work to prove Theorem \ref{thm:finitelymany}.

\begin{proof}[Proof of Theorem \ref{thm:finitelymany}]
By \cite[Corollary K]{equations}, it follows that there are at most finitely many solutions to any generalised Ramanujan--Nagell equation such as \eqref{eqn:generalisedRN} and, consequently, there are at most finitely many solutions to \eqref{eqn:RN}, showing that there can only be finitely many perfect $2-$error correcting $q-$ary codes for a fixed value of $q$.

For the convenience of the reader, we shall give a different proof of the fact that \eqref{eqn:generalisedRN} has finitely many solutions. This proof will allow us to fix notation and will be insightful for our work in Section \ref{Sec:resolve}. 

Without loss of generality, we can assume that $\gcd(c, q_1\dots q_k) = 1$. This is because any powers of $q_1, \dots, q_k$ present in $c$ can be absorbed into $y$ without altering the number of solutions of \eqref{eqn:generalisedRN}. Then, we can write 
\begin{equation}
    \label{eqn:defs1}
c = c_0c_1^3 \quad \text{and} \quad y = y_0y_1^3,
\end{equation}
where $c_0, c_1, y_0, y_1 \in \Z$ with $c_0$ and $y_0$ cubefree. We note here that 
\begin{equation}
    \label{eqn:finitelymanyd}
c_0y_0 \in \{c_0q_1^{a_1}\dots q_k^{a_k} \mid 0 \le a_1, \dots, a_k \le 2 \}.
\end{equation}
We define $d$ and $z$ by
\begin{equation}
    \label{eqn:defs2}
d = c_0y_0 \quad \text{ and } \quad z = c_1y_1,
\end{equation}
so that \eqref{eqn:generalisedRN} can be rewritten as 
\begin{equation}
    \label{eqn:eq8}
x^2 + b = dz^3.
\end{equation}
Let $(x, z) \in \Z^2$ be a solution to the previous equation and let $(X,Y)$ be given by
\begin{equation}
    \label{eqn:defs3}
X = dz \quad \text{and} \quad Y = dx. 
\end{equation}
Then, $(X,Y)$ is an integral point on the Mordell curve
\begin{equation}
    \label{eqn:Mordell}
E_d: Y^2 = X^3 - d^2b. 
\end{equation}
By the work of Siegel \cite{Siegel}, these curves have finitely many integral points. In addition, since the set in \eqref{eqn:finitelymanyd} is finite, there are only finitely many curves $E_d$ to consider. Consequently, there are at most finitely many solutions to \eqref{eqn:generalisedRN}, finishing the proof.
\end{proof}

%\begin{remark}
%    The height bounds that we used in the previous proof are not sharp, and so we should not use them in order to explicitly solve \eqref{eqn:Sunitequation}. In Section \ref{Sec:resolve}, we shall use a better algorithm to obtain all solutions to \eqref{eqn:Sunitequation} and, consequently, all solutions to \eqref{eqn:RN}.
%\end{remark}

From the definition of $x$ and $D_q$ given in \eqref{eqn:definitions}, we note that $x$ and $D_q$ are even precisely when $q$ is odd. Therefore, we can define
\begin{equation}
    \label{eqn:eq9}
x' = \frac{x}{2} \quad \text{and} \quad D'_q = \frac{D_q}{2}.
\end{equation}
In these cases, we can divide \eqref{eqn:RN} by $4$ and obtain the following equation, which, computationally, is slightly easier to solve:
\begin{equation}
   \label{eqn:RN2}
(x')^2 + D'_q = 2q_1^{n_1}\dots q_k^{n_k}.
\end{equation}
If we let $b = D'_q$ and $c = 2$, this is another instance of a generalised Ramanujan--Nagell equation in the form of \eqref{eqn:generalisedRN}. We shall present an algorithm for resolving these equations in Section \ref{Sec:resolve}.

%We will present an algorithm to solve \eqref{eqn:RN} and \eqref{eqn:RN2} in Section \ref{Sec:resolve}.

\section{Resolving the Diophantine equations}
\label{Sec:resolve}
As we explained in Section \ref{Sec:RN}, the existence of a perfect $2-$error correcting code over a $q-$ary alphabet implies the existence of a solution $(x,y)$ to the generalised Ramanujan--Nagell equation \eqref{eqn:generalisedRN}. In this section, we shall present an algorithm to solve this type of equations. Our algorithm is based on 
%$(x, n_1, \dots, n_k)$ to \eqref{eqn:RN} (if $q$ is even) or to \eqref{eqn:RN2} (if $q$ is odd). In this section, we shall use a variant of 
\cite[Algorithm 6.2]{equations}.% in order to resolve both equations. If there are no solutions for a fixed value of $q$, we can conclude that there are no perfect $2-$error correcting codes over a $q-$ary alphabet. Otherwise, we need a way to deal with the outstanding solutions, which we shall present in Section \ref{Sec:Lloyd}.

As we discussed in the proof of Theorem \ref{thm:finitelymany}, we can reduce the resolution of \eqref{eqn:generalisedRN} to the determination of all integral points $(X,Y)$ on a finite number of Mordell curves $E_d$. Then, by \eqref{eqn:defs1}, \eqref{eqn:defs2} and \eqref{eqn:defs3}, we may recover the original solutions $x$ and $y$ via the expressions
\begin{equation}
    \label{eqn:eq10}
x = \frac{Y}{d} \quad \text{and} \quad y = \frac{d}{c_0}\left(\frac{X}{c_1d}\right)^3.
\end{equation}
In order to compute the integral points on the Mordell curves $E_d$, we distinguish two cases. Following \cite[Section 4.1]{equations}, we define $a_S$ by 
\begin{equation}
    \label{eqn:eq11}
a_S = 1728\prod_{p} p^{\min\{\text{ord}_p(bd^2), 
 2\}},
\end{equation}
where the product is taken over all primes $p$ and $\text{ord}_p(.)$ denotes the standard $p-$adic valuation. If $a_S < 500,000$, we can use \cite[Algorithm 4.2]{equations} to compute all integral points very efficiently. This restriction is due to the fact that the algorithm requires the computation of all elliptic curves of conductor $N \mid a_S$. If $a_S < 500,000$, these curves have been computed by Cremona \cite{Cremona} and so the necessary computational effort is minimal.

In the case where $a_S \ge 500,000$, our code computes the integral points by using \texttt{Magma}'s function \textbf{IntegralPoints}, which is based upon the use of linear forms in complex and $p$-adic elliptic logarithms (for a reference, see \cite{IntegralPoints}). In this case, the main computational difficulty lies in the fact that a Mordell-Weil basis for $E_d$ needs to be computed which, in general, is very hard. We introduce two tricks to mitigate this difficulty.

Firstly, we note that the question of determining the rank $r$ of $E_d$ is highly non-trivial. However, by the work of Kolyvagin \cite{Kolyvagin}, we know that if the analytic rank of an elliptic curve is $0$ or $1$, its algebraic rank coincides with its analytic rank. Since computing the analytic rank is much easier computationally than computing the algebraic rank, there are substantial computational savings in these cases.

We also remark that if $r=0$, there is no need to compute any generators to the Mordell-Weil group. Similarly, if $r=1$, a generator can be found by computing Heegner points, as shown by Gross and Zagier \cite{Heegner}. We implement these two improvements in our code. While it could seem that these are not significant for the computation, they actually help tremendously and allow us to solve many equations which would otherwise not be amenable to our techniques.

Alternatively, we could have used the work of Bennett and Ghadermarzi \cite{BennettGha}, who computed all integral points on the Mordell curves
\begin{equation}
    \label{eqn:eq12}
E_k: y^2 = x^3 + k,
\end{equation}
where $0 < |k| \le 10^7$. They use a different set of techniques, involving the resolution of cubic Thue equations. This entrails a similar level of computational complexity to our approach, and it does not introduce any substantial computational improvements.

%However, the computational complexity is similar and so we have decided to use the aforementioned approach.

%, have been previously found by Bennett and by using a different set of techniques, involving cubic Thue equations.

%In order to present our algorithm, we let $S$, $N_S$, $\cO$ and $\cO^*$ be as in \eqref{eqn:SNS} and \eqref{eqn:rings}. As we discussed in the proof of Theorem \ref{thm:finitelymany}, any solution $(x, n_1, \dots, n_k)$ to \eqref{eqn:RN} or \eqref{eqn:RN2} will give rise to a solution $(x,y)$ of the equation
%\begin{equation}
%    \label{eqn:generalisedRN}
%x^2 + b = cy, \quad x \in \cO, \quad y \in \cO^*,
%\end{equation}
%where $b = D_q$ and $c = 8$ (for the case where $q$ is even) or $b = D'_q$ and $c = 2$ (if $q$ is odd). 

%By applying \cite[Algorithm 6.1]{equations}, we can reduce the resolution of \eqref{eqn:generalisedRN} to the determination of $S-$integral points on certain Mordell curves.

With the presented algorithm, we are able to solve the generalised Ramanujan--Nagell equations and recover the values of $n$ and $M$ which could represent a perfect $2-$error correcting $q-$ary code. This is the content of the following lemma.

\begin{lemma}
    \label{lemma:solutions}
    Let $q \ge 6$ be a positive integer which is not a prime power {satisfying either of the two conditions in Theorem \ref{thm:main}}. Then, all solutions $(n, M)$ to \eqref{eqn:Hamming1} with $n \ge 5$ are given in Table \ref{tab:solutions}.
\end{lemma}

\begin{proof}
    Using the algorithm that we have previously presented, we obtain all solutions $(x,y) \in \Z^2$ to each of the generalised Ramanujan--Nagell equations. By \eqref{eqn:definitions}, we can then recover $n$ as 
    \begin{equation}
        \label{eqn:eq13}
    n = \frac{x+q-3}{2(q-1)},
    \end{equation}
    and obtaining $M$ is immediate from \eqref{eqn:Hamming1}. All the solutions that we found are contained in Table \ref{tab:solutions}.
\end{proof}

\begin{table}[!ht]
    \centering 
    \begin{tabular}{c|cc}
    \hline 
    \toprule
    $q$ & $n$ & $M$ \\
    \midrule
    $15$ & $11$ & $3^4\cdot 5^{10}$ \\
    
    $21$ & $52$ & $3^{40}\cdot 7^{52}$ \\
     
    $46$ & $93$ & $2^{79}\cdot 23^{91}$ \\
    \bottomrule
    \end{tabular}
    \caption{Solutions to \eqref{eqn:Hamming1} with $n \ge 5$ and $q$ satisfying the conditions in Theorem \ref{thm:main}.}
    \label{tab:solutions}
\end{table}

\section{Dealing with the outstanding solutions: Lloyd's theorem}
\label{Sec:Lloyd}
In order to finish the proof of Theorem \ref{thm:main}, we need to show that none of the solutions in Table \ref{tab:solutions} can be parameters for a perfect $q-$ary code. For this purpose, we shall use Lloyd's Theorem, which was originally proved by Lloyd \cite{Lloyd} if $q$ is a prime power. The general case was proved independently by 
%and was subsequently extended to an arbitrary value of $q$ independently by 
Bassalygo \cite{Lloyd2}, Delsarte \cite{Lloyd3} and Lenstra \cite{Lloyd4}. The statement of that theorem is as follows.

\begin{theorem}(Lloyd)
    \label{thm:lloyd}
    Let $n, e$ and $q$ be positive integers. Suppose that there exists a perfect $e-$error correcting code over a $q-$ary alphabet with word length $n$. Then, the polynomial
    \begin{equation}
    \label{eqn:lloyd}
    L_e(x) = \sum_{i=0}^e (-1)^i(q-1)^{e-i}\binom{x-1}{i}\binom{n-x}{e-i}
    \end{equation}
    has e distinct integer zeros in the interval $[1,n]$.
\end{theorem}

Our previous work, together with Lloyd's theorem, allow us to finish the proof of Theorem \ref{thm:main}.

\begin{proof}[Proof of Theorem \ref{thm:main}]
    Let $q$ be a number satisfying the conditions in the theorem and suppose for contradiction that there exists a perfect $2-$error correcting code over a $q-$ary alphabet, with parameters $(n, M)$.
    
    By our discussion in Sections \ref{Sec:RN} and \ref{Sec:resolve}, along with Lemma \ref{lemma:solutions}, we know that the tuple $(q,n,M)$ is contained in one of the rows of Table \ref{tab:solutions}. By Theorem \ref{thm:lloyd}, it follows that the polynomial $L_2(x)$ given by 
    \begin{equation}
        \label{eqn:eq14}
    L_2(x) = (q-1)^2\frac{(n-x)(n-x-1)}{2} - (q-1)(x-1)(n-x) + \frac{(x-1)(x-2)}{2}
    \end{equation}
    has two distinct integer roots $x_1, x_2$ with $1 \le x_1 < x_2 \le n$. However, we check with \texttt{Magma} that this is not the case for any of the values of $q$ and $n$ in Table \ref{tab:solutions}. Consequently, for all the values of $q$ in the statement of Theorem \ref{thm:main}, there does not exist a perfect $2-$error correcting code over an alphabet of size $q$.
\end{proof}

\section{Limitations and future lines of work}
\label{Sec:limitationswork}
In this section, we shall briefly comment on the limitations of the methodology that we present in this paper, as well as propose future lines of work and different problems which could be solved by leveraging these techniques. 

\subsection{Limitations of our methodology}
\label{Sec:limitations}
Firstly, we remark that our methods can show non-existence of perfect $2-$error correcting codes for a fixed value of $q$, but they are not able to handle infinite families of $q$. The main reason why is due to the fact that the coefficient $D_q$ in \eqref{eqn:RN} depends on $q$ and it would no longer be constant.

For this reason, even if the set \eqref{eqn:finitelymanyd} remains finite (which could be achieved by fixing the prime divisors of $q$, for example), we would effectively need to obtain all integral points for a \textbf{family of Mordell curves} $E_{d(q)}$. In this situation, the results of Siegel \cite{Siegel} do not apply and, consequently, it is entirely possible that there are infinitely many integral points to consider. In addition, and to the best of our knowledge, there are not any techniques that allow to determine integral points over families of elliptic curves in an effective manner. This is why we believe that, in order to prove Conjecture \ref{conj:nomore}, even just for $e=2$, significantly new ideas will need to be introduced.

If we fix $q$, we note that, in principle, all the techniques that we have presented in this paper would be applicable for any $q$. However, in practice, it is computationally unfeasible to resolve \eqref{eqn:RN} for values of $q$ other than those in Theorem \ref{thm:main}. For this reason, it is unlikely that the range of values of $q$ can be extended much further with our methodology. 

This is mainly due to the fact that, despite the computational improvements that we have introduced in Section \ref{Sec:resolve}, we ultimately need to find a list of generators for the group $E_d(\Q)$, where $E_d$ is the Mordell curve \eqref{eqn:Mordell}. As the conductor of $E_d$ grows, even finding one rational point on $E_d$ is a very hard problem and, consequently, it is impractical to compute a basis for $E_d(\Q)$. 

To illustrate the computational difficulties, let us consider the case $q=94 = 2 \cdot 47$. This is excluded from Theorem \ref{thm:main} because our methodology would require us to work with the curve
\begin{equation}
    \label{eqn:eq15}
E_{2\cdot 47^2}: y^2 = x^3 + 161478403652,
\end{equation}
which has conductor {$N = 16328471028588 = 2^2\cdot3^3\cdot 47^2\cdot8273^2$}. While we are able to show that $E_{2\cdot 47^2}$ has rank $1$, $N$ is too large to apply the Heegner point methodology and, consequently, we are unable to find a generator of $E_{2\cdot 47^2}(\Q)$.

If $q$ has a large prime divisor $p$, there is a slightly different problem. In this instance, and for $q$ large, we need to use the \texttt{Magma} \textbf{IntegralPoints} subroutine to find integral points on the Mordell curve \eqref{eqn:Mordell}. However, this relies on the use of linear forms in $p-$adic elliptic logarithms (see \cite{IntegralPoints}) and the bounds obtained by applying this method are astronomical as $p$ grows, which makes it impossible to find all solutions to \eqref{eqn:RN}. In fact, this is the main reason why we needed to exclude prime divisors larger than $13$ if $q \le 600$.

Due to the high computational burden, we believe that the results that we present here cannot be significantly extended without the introduction of new algorithms to resolve the Ramanujan--Nagell equation \eqref{eqn:RN}. Promising results in this direction are the newly-improved estimates for linear forms in three logarithms by Mignotte and Voutier \cite{Voutier} and some work on the generalised Ramanujan--Nagell equation by Mutlu, Le and Soydan \cite{generalisedRN}, which, if extended, could allow for a more efficient resolution of \eqref{eqn:RN}.

\subsection{Future lines of work}
\label{Sec:futurelines}
In this paper, we used extensively two features of perfect $2-$error correcting codes: the perfect code condition \eqref{eqn:Hamming1} and the fact that Lloyd's Theorem (Theorem \ref{thm:lloyd}) applies. As long as these two conditions are satisfied, it seems plausible that our approach could generalise.

For instance, the methods that we present here could be used to approach perfect $2-$error correcting quantum codes, by mimicking this approach and changing the perfect code condition \eqref{eqn:perfectcondition} by the perfect quantum code condition:
\begin{equation}
\label{eqn:eq16}
    M = {q^n}\left/\sum_{j = 0}^e {\binom{n}{j}}(q^2-1)^j\right..
\end{equation}
Li and Xing \cite{quantum} classified perfect quantum codes over $q-$dimensional qudits, where $q$ is a prime power. In addition, they proved an analogue of Lloyd's theorem for quantum codes. This makes the situation completely analogous to the classical case and, for this reason, it seems reasonable that the approach that we presented here would directly generalise to quantum codes.

\section{Conclusions}
\label{Sec:conclusions}
The main findings of this paper are Theorems \ref{thm:main} and \ref{thm:finitelymany}, where we show that perfect $2$-error correcting codes over $q-$ary alphabets do not exist for more than $170$ new values of $q$, and show that, for a fixed value of $q$, there can only be finitely many $2-$error correcting perfect codes.

In previously existing literature, no values of $q$ with more than three prime factors were considered and, in addition, the methods used were specific to the cases at hand and could not be generalised. In addition, there had been no new results in over $50$ years, showing a clear gap in the literature.

In this paper, we successfully address this gap by showing non-existence of perfect codes for some values of $q$ with four prime factors ($q = 210, 330, 462$) and provide a method which would theoretically apply to any fixed value of $q$, with the only existing problems arising from current computational technology. Our results improve the number of $q$ for which $2-$error correcting codes have been classify by more than 2,000\%.

We remark that these results are also relevant from a practical point of view, since they show that, for many values of $q$, the optimal information transmission scheme while attempting to correct $2$ errors cannot be that of perfect codes and, consequently, alternative schemes need to be considered in order to achieve the best possible codes.

In order to verify all computations carried out in this paper, we provide the reader with the \texttt{Magma} code available in the author's GitHub repository (accessible at \href{https://github.com/PJCazorla/perfect-q-ary-codes}{https://github.com/PJ\-Cazorla/perfect-q-ary-codes}), which can be used to check the correctness of Theorem \ref{thm:main}.

Finally, we remark that our results are consistent with the existing literature in the cases where they overlap and, in addition, they support Conjecture \ref{conj:nomore}, which is widely believed to be true.


\begin{thebibliography}{}

\bibitem{Lloyd2}
L.\ A.\ Bassalygo, \emph{Generalization of Lloyd’s theorem to arbitrary alphabet [in Russian]}, Problemy Upravlenija i Teorii Informacii \textbf{2}(2), (1973), 133--137. The English translation is available in Problems of Control and Information Theory \textbf{2}(2), (1973) 25--28.

\bibitem{Bassalygo}
L.\ A.\ Bassalygo, V.\ A.\ Zinoviev, V.\ K.\ Leontiev, N.\ I.\ Feldman, \emph{Nonexistence of perfect codes over some composite alphabets (in Russian)}, Problemy Peredachi Informacii, XI \textbf{3} (1975), 3--13.

\bibitem{BennettGha}
M.\ A.\ Bennett and A.\ Ghadermarzi, \emph{Mordell's equation : a classical approach}, LMS Journal of Computation and Mathematics, \textbf{18} (2015), 633--646. 

\bibitem{Best}
M.\ R.\ Best, \emph{Perfect codes hardly exist}, IEEE Transactions on Information Theory, \textbf{29}(3) (1983), 349--351.

\bibitem{Magma}
W. Bosma,   J. Cannon and C. Playoust,
\emph{The {M}agma algebra system. {I}. {T}he user language},
Computational algebra and number theory (London, 1993),
J. Symbolic Comput. \textbf{24}(3--4) (1997),  235--265.

\bibitem{IBM}
C.\ L.\ Chen and M.\ Y.\ Hsiao, \emph{Error-Correcting Codes for Semiconductor Memory Applications: A State-of-the-Art Review}, IBM Journal of Research and Development,\textbf{28} (2), (1984), 124--134.

\bibitem{Cremona}
J.~Cremona, \emph{Algorithms for Modular Elliptic Curves}, Cambridge University Press (1997), \url{https://homepages.warwick.ac.uk/staff/J.E.Cremona/book/fulltext/index.html}.

\bibitem{Lloyd3}
P.\ Delsarte, \emph{An Algebraic Approach to the Association Schemes of Coding Theory}, Philips
Research Reports Supplements \textbf{10}, (1973), Philips Research Laboratories, Eindhoven.

\bibitem{Heegner}
B.\ H.\ Gross and D.\ B.\ Zagier, \emph{Heegner points and derivatives of $L$-series}, Invent. Math. \textbf{84} (1986), 225--320.

\bibitem{NRT}
V.\ Gubitosi, A.\ Portela and C.\ Qureshi, \emph{On the non-existence of perfect codes in the Niederreiter-Rosenbloom-Tsfasman metric}, \href{https://arxiv.org/abs/2302.11738}{https://arxiv.org/abs/2302.11738}.

\bibitem{survey}
O.\ Heden, \emph{On perfect codes over non prime power alphabets}, Contemporary Mathematics \textbf{523} (2010), 173--184.

\bibitem{Hill}
R.\ Hill, \emph{A First Course in Coding Theory}, Oxford Applied Mathematics and Computing Science Series (1986), Second Edition, Oxford.

\bibitem{Hong}
Y.\ Hong, \emph{On the nonexistence of unknown perfect $6-$ and $8-$ codes in Hamming schemes
$H(n, q)$ with $q$ arbitrary}, Osaka J. Math. \textbf{21} (1984), 687--700.

\bibitem{equations}
    R.\ von Känel and B.\ Matschke, \emph{Solving $S$-unit, Mordell, Thue, Thue–Mahler and generalized Ramanujan–Nagell equations via Shimura–Taniyama conjecture}, Memoirs of the American Mathematical Society, \textbf{286} (1419), 2016.

\bibitem{Kolyvagin}
    V.\ Kolyvagin, \emph{Finiteness of $E(\Q)$ and $X(E, Q)$ for a class of Weil curves}, Math. USSR Izv. \textbf{32} (3) (1989),523--541.

\bibitem{RNsurvey}
M.\ Le and G.\ Soydan, \emph{A brief survey on the generalized Lebesgue--Ramanujan--Nagell equation}, Surveys in Mathematics and its Applications \textbf{15} (2020), 473--523.

\bibitem{Lloyd4}
H.\ W.\ Lenstra, \emph{Two theorems on perfect codes}, Discrete Mathematics \textbf{3}, (1972) 125--132.

\bibitem{primepowers2}
V.\ K.\ Leontiev and V.\ A.\ Zinoviev, \emph{Nonexistence of perfect codes over Galois fields}, Problems of Inform. Theory, \textbf{2}(2) (1973), 123--132.

\bibitem{quantum}
Z.\ Li and L.\ Xing, \emph{Classification of q-Ary Perfect Quantum Codes}, IEEE Transactions on Information Theory \textbf{59} (1), (2013), 631--634.

\bibitem{vanLint}
J.\ H.\ van Lint, \emph{Recent results on perfect codes and related topics}, Combinatorics (ed by M.
Hall Jr and J. H. van Lint) (1974), 158--178, Mathematical Center Tracts 55, Amsterdam.

\bibitem{Lloyd}
S.\ P.\ Lloyd, \emph{Block coding}, The Bell System Technical Journal \textbf{36}, (1957), 517--535.

\bibitem{Voutier}
M.\ Mignotte and P.\ Voutier, \emph{A kit for linear forms in three logarithms}, 2023, to appear in Math. Comp.

\bibitem{Mordell}
L.\ Mordell, \emph{The Diophantine equation $x^2+7=2^n$}, Norsk Matematisk Tidsskrift \textbf{30}, (1948), 62--64.

\bibitem{generalisedRN}
E.\ K.\ Mutlu, M. Le and G. Soydan, \emph{A modular approach to the generalized Ramanujan–Nagell equation}, Indagationes Mathematicae, \textbf{33} (5), 2022, 992--1000.

\bibitem{Reuvers}
H.\ F.\ H.\ Reuvers, \emph{Some nonexistence theorems for perfect codes over arbitrary alphabets}, PhD Thesis, Eindhoven Technological Univ. 1977.

\bibitem{Shannon}
C.\ E.\ Shannon. \emph{A mathematical theory of communication}, Bell System Tech. J. \textbf{27}, (1948), 379--423.

\bibitem{Siegel}
C.\ L.\ Siegel, \emph{Einige Anwendungen diophantischer Approximationen}, Abh. Preuss. Akad.
Wiss. Phys. Math. Kl. \textbf{1} (1929), 41--69.

\bibitem{IntegralPoints}
R.\ J.\ Stroeker and N.\ Tzanakis, \emph{Solving elliptic diophantine equations by estimating linear forms in elliptic logarithms}, Acta Arith., \textbf{67} (1994), 177--196.

\bibitem{primepowers}
A.\ Tietäväinen, \emph{On the nonexistence of perfect codes over finite fields}, SIAM J. Appl. Math. \textbf{24} (1973), 88--96.

\end{thebibliography}
\end{document}